\documentclass[12pt]{article}
\usepackage{amsthm,amssymb,amsmath}
\usepackage{epsfig}
\usepackage{fullpage}

\title{Median eigenvalues and the HOMO-LUMO index of graphs}

\author{Bojan Mohar\thanks{Supported in part by the
  Research Grant P1--0297 of ARRS (Slovenia), by an NSERC Discovery Grant (Canada)
  and by the Canada Research Chair program.}~\thanks{On leave from:
  IMFM \& FMF, Department of Mathematics, University of Ljubljana, Ljubljana,
  Slovenia.}\\
  {Department of Mathematics}\\
  {Simon Fraser University}\\
  {Burnaby, B.C. V5A 1S6} \\
  email: {\tt mohar@sfu.ca}
}

\newtheorem{theorem}{Theorem}[section]
\newtheorem{lemma}[theorem]{Lemma}

\newtheorem{conjecture}[theorem]{Conjecture}

\newcommand{\DEF}[1]{{\em #1\/}}

\renewcommand\l{\lambda}
\begin{document}

\maketitle

\begin{abstract}
Motivated by the problem about HOMO-LUMO separation that arises in mathematical chemistry, Fowler and Pisanski \cite{FP1,FP2} introduced the notion of the HL-index which measures how large in absolute value may be the median eigenvalues of a graph. In this note we provide rather tight lower and upper bounds on the maximum value of the HL-index among all graphs with given average degree. In particular, we determine the exact value of this parameter when restricted to chemically relevant graphs, i.e.\ graphs of maximum degree $3$, and thus answer a question from \cite{FP1,FP2,FJP}.
The proof provides additional insight about eigenvalue distribution of large subcubic graphs.
\end{abstract}

\section{Introduction}

In a recent work, Fowler and Pisanski \cite{FP1,FP2} (see also Jakli\v c et al.\ \cite{FJP}) introduced the notion of the \emph{HL-index} of a graph that is related to the HOMO-LUMO separation studied in theoretical chemistry. This is the gap between the Highest Occupied Molecular Orbital (HOMO) and Lowest Unoccupied Molecular Orbital (LUMO). The energies of these orbitals are in linear relation with eigenvalues of the corresponding molecular graph and can be expressed as follows. Let $G$ be a (molecular) graph of order $n$, and let
$\lambda_1\ge \lambda_2\ge \cdots \ge\lambda_n$ be the eigenvalues of its adjacency matrix. The eigenvalues occurring in the HOMO-LUMO separation are $\lambda_H$ and $\lambda_L$, where
$$
   H=\lfloor \tfrac{n+1}{2}\rfloor \qquad \textrm{and} \qquad
   L=\lceil \tfrac{n+1}{2}\rceil.
$$
The \emph{HL-index\/} $R(G)$ of the graph $G$ is then defined as
$$R(G) = \max \{|\lambda_H|,|\lambda_L|\}.$$

In \cite{FP1,FP2} it is proved that for every ``chemical'' graph, i.e., a graph $G$ of maximum degree at most 3, we have $0\le R(G) \le 3$ and that if $G$ is bipartite, then $R(G)\le \sqrt{3}$. Let $R(3) := \sup \{R(G)\}$, where the supremum is taken over all graphs $G$ whose maximum degree $\Delta(G)$ is at most 3. It is proved in \cite{FJP} that $\sqrt{2} \le R(3) \le 3$.

We can define a similar quantity $R(d)$ for every positive real value $d$:
$$R(d) = \sup \{R(G)\mid \Delta(G)\le d\}.$$
Similarly, let
$$\widehat R(d) = \sup \{R(G)\mid \bar d(G)\le d\}$$
where $\bar d(G)=|E(G)|/|V(G)|$ denotes the average degree of $G$.
Note that $\widehat R(d) \ge R(d)$ for every $d$.
In this paper we provide lower and upper bounds on the values $R(d)$ and $\widehat R(d)$ for every $d\ge0$ (Theorem \ref{thm:average degree}, and we determine $R(3)$ exactly (Theorem \ref{thm:cubic}). The proof, which is based on interlacing, has some consequences of additional interest (see Section \ref{sect:large}) about eigenvalues of large subcubic graphs.

\section{The general case}
\label{sect:2}

\begin{theorem}
\label{thm:average degree}
For every\/ $d>0$ we have $\widehat R(d) \le \sqrt{d}$. If $d$ is an integer that can be written as $p^k+1$, where $p$ is a prime number and $k\ge 1$ is an integer, then $R(d)\ge \sqrt{d-1}$.
\end{theorem}

\begin{proof}
To prove the lower bound, we take the incidence graph of the projective plane of order $p^k=d-1$. This is a bipartite graph whose vertices correspond to points and lines of the projective plane $PG(2,p^k)$, and two of them are adjacent if they are a point and a line incident to each other. This graph has degree $d$ and eigenvalues $\pm d$ and $\pm \sqrt{d-1}$ (cf., e.g., \cite{GR}), and its HL-index is $\sqrt{d-1}$.

To prove the upper bound, we use the following facts. Let $G$ be a graph with average degree at most $d$. Let $\lambda_1\ge \lambda_2\ge \cdots \ge\lambda_n$ ($n=|V(G)|$) be its eigenvalues. Since the adjacency matrix $A$ of $G$ has zeros on the diagonal, we have that
\begin{equation}
\sum_{i=1}^n \lambda_i = tr(A) = 0.
\label{eq:1}
\end{equation}
The diagonal entries of the matrix $A^2$ are precisely the vertex degrees. Thus,
\begin{equation}
\sum_{i=1}^n \lambda_i^2 = tr(A^2) = \sum_{v\in V(G)} \deg(v) \le nd.
\label{eq:2}
\end{equation}

Let us first assume that $R(G)=\lambda_H>0$. Let $l>H$ be the smallest integer such that $\lambda_l<0$. Then we define $C=\{1,\dots,H\}$ and $D=\{l,l+1,\dots,n\}$.
If $R(G)\ne\lambda_H$ or $\lambda_H\le0$, then $R(G)=|\lambda_L|$ and $\lambda_L\le0$. In this case we define $l$ as the largest index such that $\lambda_l>0$ and we set $C=\{L,L+1,\dots,n\}$ and $D=\{1,\dots,l\}$. Clearly, we have
\begin{equation}
\sum_{i\in C} \lambda_i^2 \ge |C|\, R(G)^2.
\label{eq:3}
\end{equation}
By applying the Cauchy-Schwartz inequality, we derive the following bound
\begin{equation}
\sum_{i\notin C} \lambda_i^2\ge \sum_{i\in D} \lambda_i^2\ge
\frac{1}{|D|}\Bigl(\sum_{i\in D} |\lambda_i|\Bigr)^2.
\label{eq:4}
\end{equation}
By (\ref{eq:1}) we have
\begin{equation}
\sum_{i\in D} |\lambda_i|\ge \sum_{i\in C} |\lambda_i| \ge |C|\,R(G).
\label{eq:5}
\end{equation}
Finally, combining the above inequalities, we get
\begin{eqnarray}
nd &\ge \sum_{i=1}^n \lambda_i^2 \ge \sum_{i\in C} \lambda_i^2 + \sum_{i\in D} \lambda_i^2 \nonumber\\[2mm]
&\ge |C|\, R(G)^2 + \frac{1}{|D|}(|C|R(G))^2 \label{eq:added6}\\[2mm]
&\ge \tfrac{n}{2}\,R(G)^2 + \tfrac{n}{2}\,R(G)^2 = nR(G)^2. \nonumber
\end{eqnarray}
This implies the stated upper bound on $\widehat R(d)$, and completes the proof.
\end{proof}

The upper bound of $\sqrt{d}$ in Theorem \ref{thm:average degree} can actually be improved by involving dependence on $n=|V(G)|$.
All bounds (\ref{eq:3})--(\ref{eq:added6}) hold for the modified sequence $\l_i'$ ($1\le i\le n$), where we replace the largest eigenvalue $\l_1$ by $\l_1'=R(G)$ and put $\l_i'=\l_i$ otherwise. Then (\ref{eq:2}) and the fact that $\l_1\ge \bar d(G)$ imply that
$$
   nd\ge d^2 - R(G)^2 + \sum_{i=1}^n \l_i'^2 \ge
   (n-1)R(G)^2 + d^2
$$
if $d=\bar d(G)$. This yields the following upper bound
$$
   R(G) \ge \sqrt{d - \tfrac{d(d-1)}{n-1}}.
$$
In particular, if a graph with average degree $d$ has at most $d^2-d+1$ vertices, then $R(G)\le\sqrt{d-1}$.

\section{Chemically relevant graphs ($d\le 3$)}

Theorem \ref{thm:average degree} restricted to the case when $d=3$ gives the bound
$\sqrt{2} \le R(3) \le \sqrt{3}$. However, the upper bound can be improved by a more elaborate technique. We will combine combinatorial arguments and interlacing of eigenvalues to determine the exact value of $R(3)$.

\begin{theorem}
\label{thm:cubic}
The median eigenvalues $\l_H$ and $\l_L$ of every subcubic graph are contained in the interval $[-\sqrt{2},\sqrt{2}\,]$.
Consequently, $R(3) = \sqrt{2}$.
\end{theorem}

In order to prove the theorem, we need some preparation. Let us first state the eigenvalue interlacing theorem (cf., e.g., \cite{GR}) that we shall use in the sequel. For a graph $G$, we let $\l_i(G)$ be the $i$th largest and $\l^-_i(G)$ be the $i$th smallest eigenvalue of $G$ (counting multiplicities).

\begin{theorem}
\label{thm:interlacing}
Let\/ $A\subset V(G)$ be a vertex set of cardinality $k$, and let $K=G-A$. Then for every $i=1,\dots,n-k$, we have
$$
    \l_i(G)\ge \l_i(K)\ge \l_{i+k}(G) \qquad \textrm{and} \qquad
    \l^-_i(G)\le \l^-_i(K)\le \l^-_{i+k}(G).
$$
\end{theorem}

\begin{figure}[htb]
   \centering
   \includegraphics[width=9cm]{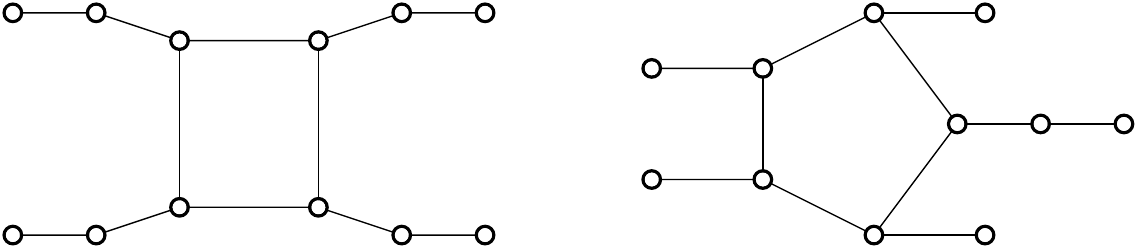}
   \caption{The graphs $C_4(2,2,2,2)$ and $C_5(2,1,1,1,1)$}
   \label{fig:1}
\end{figure}

In estimating the eigenvalues, we shall need the following facts. If $P$ is a path of length at most 2, then $\l_1(P)=-\l^-_1(P)\le\sqrt{2}$. For $k\ge 3$ and non-negative integers $t_1,\dots t_k$, let $C_k(t_1,\dots,t_k)$ be the graph obtained from the $k$-cycle by adding, for $i=1,\dots,k$, a pendent path of length $t_i$ starting at the $i$th vertex of the cycle. See Figure \ref{fig:1} for an example.
The following lemma is easy to check by computer (by Theorem \ref{thm:interlacing}, it only needs to be checked for $C_4(2,2,2,2)$ and $C_5(2,1,1,1,1)$ since other graphs in the lemma are induced subgraphs of these two).

\begin{lemma}
\label{lem:C521111}
{\rm (a)}
Let\/ $G=C_4(t_1,t_2,t_3,t_4)$, where $0\le t_i\le 2$ for $i=1,\dots,4$. Then $\l_2(G)=-\l_2^-(G)\le\sqrt{2}$.

{\rm (b)}
Let\/ $G = C_5(t_1,\dots,t_5)$ where $0\le t_1\le 2$ and $0\le t_i\le 1$ for $i=2,\dots,5$. Then $\l_3(G)<\sqrt{2}$ and $\l^-_3(G) > -\sqrt{2}$.
\end{lemma}

\begin{figure}[htb]
   \centering
   \includegraphics[width=8cm]{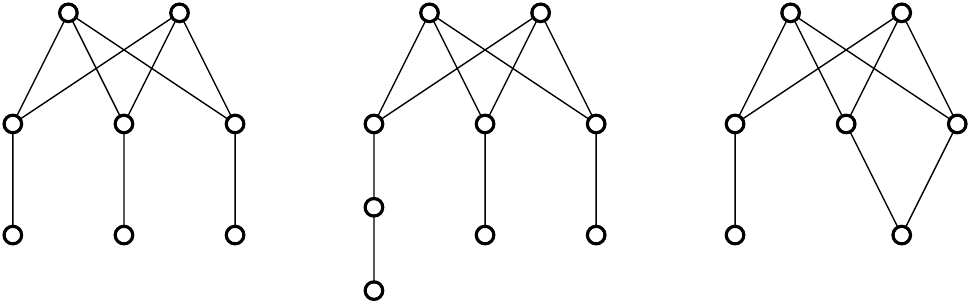}
   \caption{Some graphs containing $K_{2,3}$}
   \label{fig:2}
\end{figure}

We will need a similar result for some graphs containing $K_{2,3}$ and for some graphs containing cycles of length 3 and 4 (or 3 and 5) sharing an edge.

\begin{lemma}
\label{lem:K23}
{\rm (a)}
Let\/ $G$ be any one of the graphs shown in Figure \ref{fig:2}. Then $\l_2(G)=-\l_2^-(G)\le\sqrt{2}$.

{\rm (b)}
Let\/ $G$ be a graph isomorphic to any of the graphs shown in Figure \ref{fig:3}. Then $\l_3(G)<\sqrt{2}$ and $\l^-_3(G) > -\sqrt{2}$.

{\rm (c)}
Let\/ $G$ be a graph isomorphic to any of the graphs shown in Figure \ref{fig:4}. Then $\l_3(G)<\sqrt{2}$ and $\l^-_3(G) > -\sqrt{2}$.
\end{lemma}

\begin{figure}[htb]
   \centering
   \includegraphics[width=11cm]{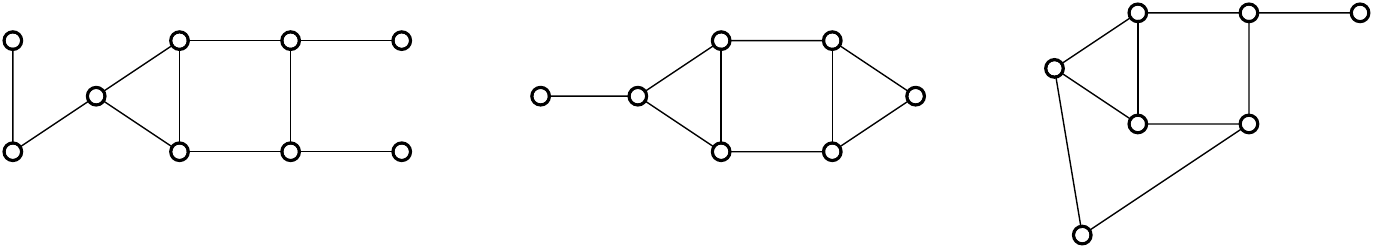}
   \caption{Three graphs containing 3- and 4-cycle sharing an edge}
   \label{fig:3}
\end{figure}

\begin{figure}[htb]
   \centering
   \includegraphics[width=6.3cm]{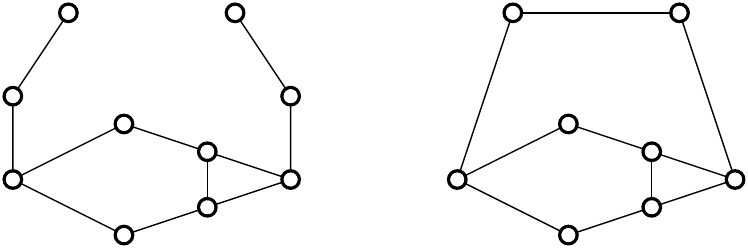}
   \caption{Two graphs containing 3- and 5-cycle sharing an edge}
   \label{fig:4}
\end{figure}

Let us recall that a graph $G$ is said to be \DEF{subcubic} if $\Delta(G)\le 3$.
A partition $\{A,B\}$ of vertices of $G$ is called \DEF{unfriendly} if every vertex in $A$ has at least as many neighbors in $B$ as in $A$, and every vertex in $B$ has at least as many neighbors in $A$ as in $B$.
A partition $\{A,B\}$ of $V(G)$ is \DEF{unbalanced} if $|A|\ne|B|$, and it is a \DEF{nice partition} if the subgraphs induced on each of the parts are disjoint unions of paths of length at most 2. We will use the following lemma:

\begin{lemma}
\label{lem:nice partition}
If\/ $G$ is a subcubic graph with an unbalanced nice partition, then $R(G)\le\sqrt{2}$.
\end{lemma}

\begin{proof}
Let $\{A,B\}$ be a nice partition with $|B|>|A|$. Since $G(B)$ is composed of paths of length at most 2, we have $\l_1(G(B))=|\l_{|B|}(G(B))| \le \l_1(P_3)=\sqrt{2}$. Since $G(B)=G-A$ is obtained from $G$ by deleting $|A|$ vertices and $|A|<H$, the eigenvalue interlacing theorem shows that $\l_H(G)\le \l_{|A|+1}(G)\le \l_1(G(B))\le \sqrt{2}$. Similarly, interlacing of smallest eigenvalues gives $\l_L(G)\ge \l_{n-|A|}(G)\ge \l_{|B|}(G(B))\ge -\sqrt{2}$.
This implies that $-\sqrt{2}\le \l_L(G)\le \l_H(G)\le \sqrt{2}$ and thus $R(G)\le\sqrt{2}$.
\end{proof}

Observe that every unfriendly partition of a subcubic graph is nice (but not necessarily unbalanced). Unfriendly partitions are easy to find.

\begin{lemma}
\label{lem:unfriendly partition}
Every multigraph has an unfriendly partition.
\end{lemma}

\begin{proof}
Let us consider a partition of $V(G)$ such that the number of edges between $A$ and $B$ is maximum possible. It is easy to see that the partition is unfriendly.
\end{proof}

Unfortunately, some graphs have no unbalanced unfriendly partition, and for some graphs, every nice partition is balanced. Examples of such graphs are $K_4$ and the prism $K_3\Box K_2$.

Let $\{A,B\}$ be a nice partition of $V(G)$. Suppose that $v$ is a vertex in $A$.
If $\{A\setminus\{v\},B\cup\{v\}\}$ is also a nice partition, then we say that $v$ is \DEF{unstable} (with respect to the partition $\{A,B\}$).
If $v\in B$, its stability is defined analogously.

\begin{lemma}
\label{lem:flippable}
Suppose that\/ $\{A,B\}$ is an nice partition of a subcubic graph $G$ such that every component of $G(B)$ has at most one edge. If a vertex $v\in A$ has at most one neighbor in $B$, then $v$ is unstable.
\end{lemma}

\begin{proof}
By removing $v$ from $A$ and putting it in $B$, we obtain a partition where $G(A)$ still consists of components with at most two edges. The only component in $G(B)$ possibly having more than one edge is the component containing $v$. However, since $v$ has at most one neighbor in $B$, that component has at most two edges.
\end{proof}

We are ready for the proof of our main theorem.

\begin{proof}[Proof of Theorem \ref{thm:cubic}]
By Theorem \ref{thm:average degree}, we only have to argue about the first statement of the theorem.
So, let us assume $G$ is a subcubic graph of order $n\ge3$. Our goal is to find an unbalanced nice partition and apply Lemma \ref{lem:nice partition}. We start by an unfriendly partition $\{A,B\}$ of $G$. If $|A|\ne|B|$ (and in particular if $n$ is odd), we have an unbalanced nice partition and we are done. Similarly, if there is an unstable vertex with respect to this partition, one of the two partitions is unbalanced, and we are done.

If $v$ is a vertex of degree at most 1, then $v$ is unstable. If $v$ is a vertex of degree 2, and its neighbors are $u,w$, we start with an unfriendly partition $\{A',B'\}$ of the multigraph $G'=(G-v)+uw$. Note that by adding the edge, we may have created a double edge. Now we add $v$ to $A'$ to obtain a nice partition of $V(G)$. If $u$ and $w$ are in different parts of the partition, then $v$ is clearly unstable. If they are in the same part, then $uw$ is not a double edge in $G'$; thus $u$ and $w$ are non-adjacent in $G$ and they have no neighbors in the same part. Thus $v$ is unstable also in this case. So, we may assume that every vertex of $G$ has degree three.

Suppose that $G$ contains a vertex $v$ such that the shortest cycle containing $v$ has length at least 6. Let $x,y,z$ be the neighbors of $v$, and let $x',x'',y',y'',z',z''$ be the vertices different from $v$ such that $x',x''$ are neighbors of $x$, vertices $y',y''$ are neighbors of $y$, and $z',z''$ are neighbors of $z$. Let $G'=(G-\{v,x,y,z\})+\{x'x'',y'y'',z'z''\}$; if $x'x''\in E(G)$ then we create parallel edges, and similarly for $y'y''$ and $z'z''$. Let $\{A',B'\}$ be an unfriendly partition of $G'$. If two vertices in $\{x',x'',y',y'',z',z''\}$ are in the same part, then they have no common neighbor in the same part (since the partition is unfriendly) and they are not adjacent (since the only possible edges among these vertices are $x'x'',y'y'',z'z''$ and they give rise to parallel edges which force their ends to be in different parts). Therefore, adding each of $x,y,z$ into any part of the partition produces a nice partition of $G-v$. Consequently, $\{A'\cup\{v\},B'\cup\{x,y,z\}\}$ and $\{A'\cup\{x,y,z\},B'\cup\{v\}\}$ are nice partitions of $G$. Clearly, at least one of them is unbalanced.

Suppose now that $G$ contains a cycle $C=v_1v_2v_3v_4v_5$ of length 5 and that the respective neighbors $u_i$ of $v_i$ ($i=1,\dots,5$) are all distinct and do not belong to $C$. In this case, let $G'=(G-V(C))+\{u_2u_4,u_3u_5\}$ (possibly creating parallel edges). Let $U=\{u_1,\dots,u_5\}$. Consider an unfriendly partition $\{A',B'\}$ of $G'$; we may assume that $|A'\cap U|\le 2$. If $|A'\cap U|=2$, then we extend the partition to a nice partition $\{A,B\}$ of $V(G)$ by putting $v_i$ ($1\le i\le5$) into $A$ if $u_i\in B'$, and putting it into $B$, otherwise. If $A'\cap U=\{u_i,u_{i+1}\}$ (indices taken modulo 5), then $v_{i+3}$ is unstable by Lemma \ref{lem:flippable}. By symmetry we may henceforth assume that either $A'\cap U=\{u_1,u_3\}$ or $A'\cap U=\{u_2,u_4\}$. In the former case, $v_4$ is unstable: by moving $v_4$ into $B$, we create a $P_3$ in $G(B)$ consisting of vertices $v_3,v_4,u_4$. But the component cannot be larger since $u_2u_4\notin E(G)$. To see this, recall that we have added the edge $u_2u_4$ into $G'$. Since $u_2$ and $u_4$ both ended up in $B'$, we have not created parallel edges, thus $u_2u_4\notin E(G)$. In the latter case, $v_5$ is unstable by a symmetric argument.

Suppose now that  $A'\cap U=\{u_i\}$. By symmetry, we may assume that $i\le3$. We set $A=A'\cup\{v_{i+1},v_{i+2},v_{i+4}\}$, and $B=B'\cup\{v_i,v_{i+3}\}$ (indices modulo 5). It is easy to see that this is a nice partition of $G$. Similarly as above, we see that $u_{i+1}$ has no neighbor in $B$, and therefore $v_{i+1}$ is unstable.

Finally, suppose that $U\subseteq B'$. Note that in this case $u_1,\dots,u_5$ are mutually nonadjacent since the added edges in $G'$ prevent further edges between these vertices, and double edges in $G'$ were not possible in this case. Moreover, $u_2,u_3,u_4,u_5$ have no neighbors in $B'$. As before, there is a nice partition formed by $A=A'\cup\{v_1,v_2,v_4\}$, and $B=B'\cup\{v_3,v_5\}$. If it is balanced, then we know that $|B'|=|A'|-1$.
In this case we form another partition as follows. We take $A=A'$ and $B=B'\cup\{v_1,\dots,v_5\}$.
Note that $|B|=|A|+4$ and that all components of $G(B)$ are paths of length at most 2, except for the component $Q$ containing $C$. Observe that $Q$ is either isomorphic to the graph $C_5(2,1,1,1,1)$ or $C_5(1,1,1,1,1)$. By Lemma \ref{lem:C521111}(b), we conclude that $\l_3(G(B))\le\sqrt{2}$ and $\l^-_3(G(B))\ge -\sqrt{2}$. Theorem \ref{thm:interlacing} now implies that $$\l_H(G)\le\l_{|A|+3}(G)\le\l_3(G(B))\le\sqrt{2}$$
and similarly
$$\l_L(G)\ge\l^-_{|A|+3}(G)\ge\l^-_3(G(B))\ge -\sqrt{2}$$
as claimed. This completes the case of the 5-cycle with distinct neighbors.

Suppose now that $G$ contains a cycle $C=v_1v_2v_3v_4$ of length 4 and that the respective neighbors $u_i$ of $v_i$ ($i=1,\dots,4$) are all distinct and do not belong to $C$. As in the case of the 5-cycle, we let $G'=(G-V(C))+\{u_1u_2,u_3u_4\}$ (possibly creating parallel edges) and $U=\{u_1,\dots,u_4\}$. Consider an unfriendly partition $\{A',B'\}$ of $G'$ where $|A'\cap U|\le 2$. In each case, we can extend this to a nice partition $\{A,B\}$ of $V(G)$, where $A'\subset A$ and $B'\subset B$ and $|A\cap V(C)|\ge 2$. We are done if this partition is unbalanced. Otherwise, we consider the partition $\{A',B'\cup V(C)\}$. Similarly as in the case of the 5-cycle, we see that every component of $G(B)$, except the component $Q$ containing $C$, contains at most one edge. Observe that $Q$ is isomorphic to an induced subgraph of $C_4(2,2,2,2)$, and by Lemma \ref{lem:C521111}(a), $\l_2(Q)\le\sqrt{2}$ and $\l_2^-(Q)\ge -\sqrt{2}$. Since $|A'|\le |B'\cup V(C)|-4$, the eigenvalue interlacing (Theorem \ref{thm:interlacing}) shows that $R(G)\le\sqrt{2}$.

Suppose now that $G$ has an edge $uv$ that is contained in two triangles, $uvw$ and $uvz$. If $zw\in E(G)$, then $G=K_4$ and $R(G)=1$. Otherwise, let $z'$ be the neighbor of $z$ different from $u,v$, and let $w'$ be the neighbor of $w$ different from $u,v$. Let $G'$ be the graph obtained from $G-\{u,v,w,z\}$ by identifying $z'$ and $w'$. (The new vertex is of degree 4 unless $z'=w'$ already in $G$, or when $z'w'\in E(G)$, when we remove the resulting loop and obtain a vertex of degree 2.) Let us take an unfriendly partition of $G'$. Observe that this partition gives rise to a nice partition $\{A',B'\}$ of $G-\{u,v,w,z\}$. We may assume that $z',w'\in A'$. Then $\{A'\cup\{u,v\},B'\cup\{z,w\}\}$ is a nice partition of $G$. The vertex $u$ is unstable, so there is a nice unbalanced partition.

To summarize, we may henceforth assume that $G$ is a cubic graph in which each vertex belongs to a cycle of length at most 5, every induced 4-cycle or 5-cycle $C$ has two vertices with a common neighbor outside $C$, and no edge is contained in two triangles. The last condition implies that every 4-cycle is induced.

\begin{figure}[htb]
   \centering
   \includegraphics[width=6.7cm]{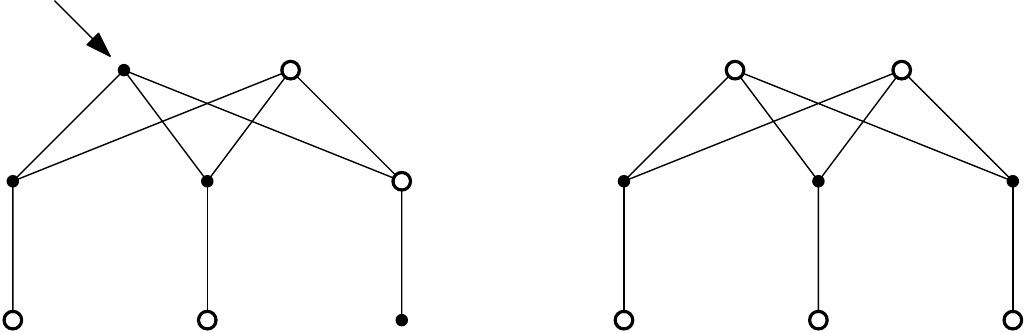}
   \caption{Extending unfriendly partition of $G'$ to $G$}
   \label{fig:4a}
\end{figure}

Suppose now that $C=v_1\dots v_4$ is an (induced) 4-cycle in $G$, and suppose that $u\notin V(C)$ is a common neighbor of $v_1$ and $v_3$. Let $a,b,c\notin \{v_1,v_3\}$ be the neighbors of $v_2,v_4,u$ (respectively). Since every 4-cycle is induced, none of $a,b,c$ are in $\{v_2,v_4,u\}$. If $a=b=c$, then $G$ is isomorphic to $K_{3,3}$, and $R(G)=0$. By symmetry we may assume that $a\neq b,c$. Let $G'$ be the graph obtained from $G$ by deleting $V(C)$ and $u$ and adding the edge $ab$. Let $\{A',B'\}$ be an unfriendly partition of $G'$, where $|A'\cap\{a,b,c\}|\le1$. This partition can be extended to a nice partition $\{A,B\}$ of $V(G)$ (i.e. $A'\subset A$, $B'\subset B$) as shown in Figure \ref{fig:4a}. The two cases shown are for the cases, where two or three of the vertices $a,b,c$ (counted twice if $b=c$) are in $B'$. The bigger white vertices represent $B$ and smaller black vertices represent the elements of $A$. In the first case in Figure \ref{fig:4a}, the vertex pointed to with an arrow is unstable, so we are done in that case. In the second case, we are done if this partition is unbalanced. Otherwise, we consider the partition $\{A',B'\cup V(C)\cup\{u\}\}$. Similarly as above, we see that every component of $G(B)$, except the component $Q$ containing $C$, has at most one edge. Observe that $Q$ is isomorphic to one of the graphs depicted in Figure \ref{fig:2}. By Lemma \ref{lem:K23}(a), $\l_2(Q)\le\sqrt{2}$ and $\l_2^-(Q)\ge -\sqrt{2}$. Since $|A'| = |B'\cup V(C)\cup\{u\}|-6$, eigenvalue interlacing (Theorem \ref{thm:interlacing}) shows that $R(G)\le\sqrt{2}$.

\begin{figure}[htb]
   \centering
   \includegraphics[width=6.5cm]{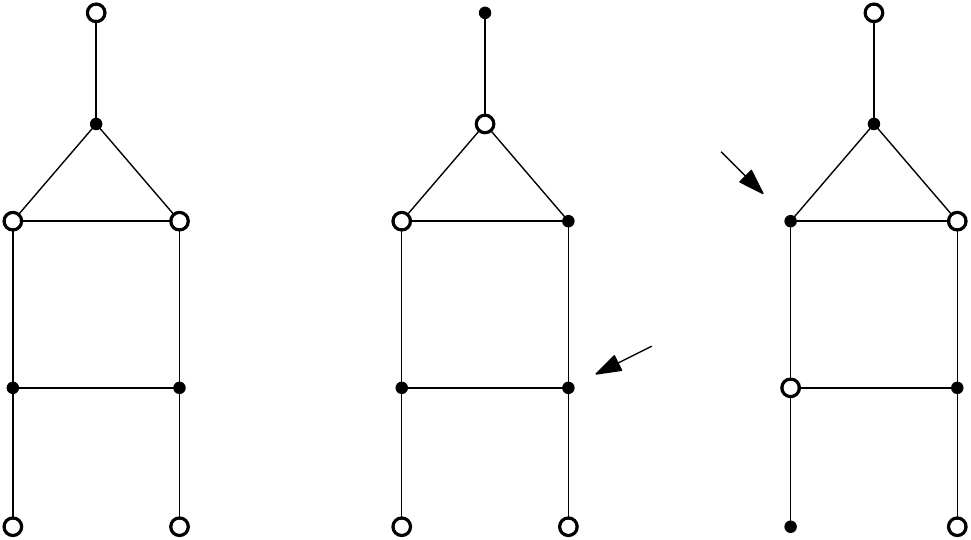}
   \caption{Extending unfriendly partition of $G'$ to $G$}
   \label{fig:4b}
\end{figure}

Let us now suppose that $u\notin V(C)$ is a common neighbor of $v_1$ and $v_2$ and let $a,b,c\notin V(C)\cup\{u\}$ be the neighbors of $v_3,v_4,u$ (respectively). If $a=b=c$, then $G$ is isomorphic to the prism $K_3\Box K_2$ that has $R(G)=0$. Otherwise, we may assume that $a\ne c$. As above, we let $G'$ be the graph obtained from $G-(V(C)\cup\{u\})$ by adding the edge $ab$ (if $a\ne b$) or $ac$ (if $a=b$). Starting with an unfriendly partition $\{A',B'\}$ of $G'$, where $|A'\cap\{a,b,c\}|\le1$, the partition can be extended to a nice partition $\{A,B\}$ of $V(G)$, where $A'\subset A$, $B'\subset B$ and $|A\cap (V(C)\cup\{u\})|=3$. See Figure \ref{fig:4b} for different possibilities. The last two cases in the figure have unstable vertices (indicated by arrows), thus we may assume that none of $a,b,c$ is in $A'$.
We are done if the partition is unbalanced. Otherwise, we consider the partition $\{A',B'\cup V(C)\cup\{u\}\}$. Similarly as above, we see that every component of $G(B)$, except the component $Q$ containing $C$, contains at most one edge. Observe that $Q$ is isomorphic to an induced subgraph of one of the graphs in Figure \ref{fig:3}. (If $a,b,c$ are distinct, we obtain an induced subgraph of the first case in the figure; if $a=b$, we obtain the second one; since we have added the edge $ac$ to $G'$, the edge $ac$ is not in $Q$; if $b=c$, then we obtain the third graph in Figure \ref{fig:3}.) By Lemma \ref{lem:K23}(b), $\l_3(Q)\le\sqrt{2}$ and $\l_3^-(Q)\ge -\sqrt{2}$. Since $|A|=|B|-6$, the eigenvalue interlacing theorem \ref{thm:interlacing} shows that $R(G)\le\sqrt{2}$.

From now on, we may assume that $G$ has no 4-cycles.

\begin{figure}[htb]
   \centering
   \includegraphics[width=9cm]{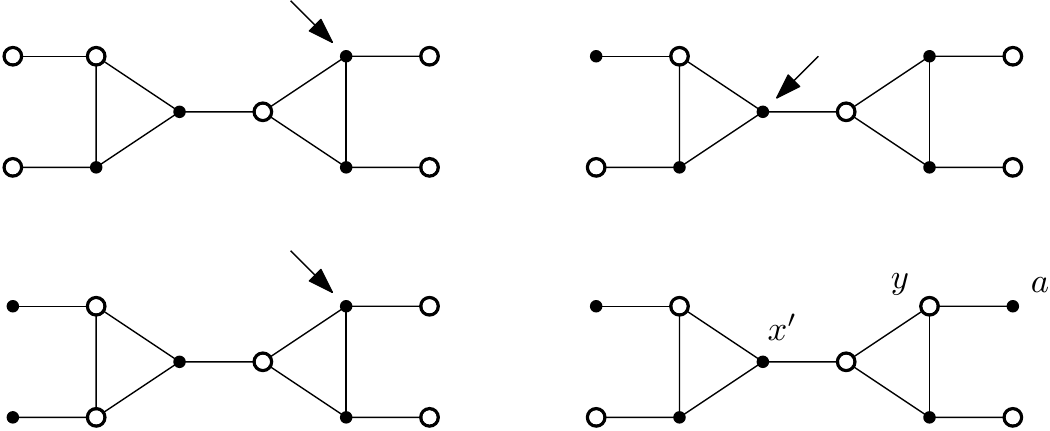}
   \caption{Extended unfriendly partition with an unstable vertex}
   \label{fig:5}
\end{figure}

Suppose that $G$ has a triangle $T=xyz$ and that the neighbor $x'\notin V(T)$ of $x$ also lies in a triangle, $T'=x'y'z'$. Excluding previously treated cases, the subgraph $K$ of $G$ induced on $T\cup T'$ consists of the two triangles together with the edge $xx'$. Let $a,b,c,d\notin V(K)$ be the respective neighbors of $y,z,y',z'$. Note that $a\ne b$ and $c\ne d$, but it may happen that $\{a,b\}\cap\{c,d\} \ne \emptyset$. Let $G'$ be the graph obtained from $G-V(K)$ by adding the edges $ab$ and $cd$. Let $\{A',B'\}$ be an unfriendly partition of $G'$, where $|A'\cap\{a,b,c,d\}|\le |B'\cap\{a,b,c,d\}|$. This partition can be extended to an unfriendly partition $\{A,B\}$ of $V(G)$ which has an unstable vertex. All possible cases (up to symmetries) are depicted in Figure \ref{fig:5}, where vertices in $A$ are shown by smaller full circles, vertices in $B$ by larger white circles. The arrows show unstable vertices in three of the four cases. In the last case, the vertex $y$ is unstable if $a$ has no neighbor in $A$. If $a$ has a neighbor in $A$, then we move $a$ from $A$ into $B$ and move $y$ from $B$ into $A$. This gives another nice partition, in which the vertex $x'$ is unstable.

\begin{figure}[htb]
   \centering
   \includegraphics[width=9cm]{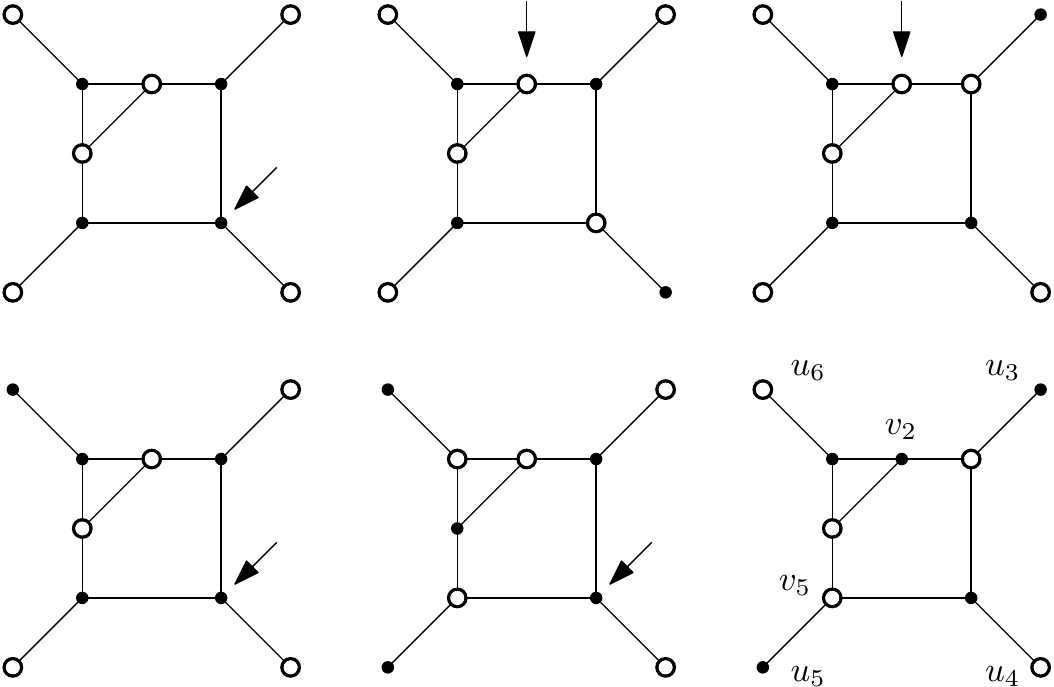}
   \caption{Extended nice partition with an unstable vertex}
   \label{fig:6}
\end{figure}

Finally, suppose that $C=v_1\dots v_5$ is a 5-cycle in $G$. Since $G$ has no 4-cycles, $C$ is induced. Let $u_i\notin V(C)$ be a neighbor of $v_i$ ($1\le i\le 5$). Two of these neighbors must coincide, and since $G$ has no 4-cycles, we may assume that $u_1=u_2$. Excluding the above case of adjacent triangles, we see that $u_3,u_4,u_5$ are all distinct. Let $u_6$ be the neighbor of $u_1$ that is not on $C$. Note that $u_6\notin\{u_3,u_5\}$ but possibly $u_6=u_4$. In this case, we let $G'=(G-(V(C)\cup\{u_1\}))+\{u_3u_4,u_6u_5\}$. We start with an unfriendly partition of $G'$ and, as above, extend it to a nice partition with an unstable vertex. See Figure \ref{fig:6} for details. This is possible in all cases except the one shown as the last one in the figure. Considering this remaining case, let us first assume that $u_4\ne u_6$. If the partition in that case is balanced, we add $V(C)\cup\{u_1\}$ into $B$. Since we have increased the cardinality of $B$ by three, it suffices to see that the component $Q$ of $G(B)$ containing $C$ satisfies $\l_3(Q)\le\sqrt{2}$ and $\l_3^-(Q)\ge -\sqrt{2}$. This now follows by Lemma \ref{lem:K23}(c) since $Q$ is isomorphic to an induced subgraph of one of the graphs shown in Figure~\ref{fig:4}. Suppose now that $u_4=u_6$. If $u_5$ has no neighbor in $A$, then $v_5$ is unstable. Otherwise, we move $u_5$ from $A$ into $B$ and move $v_5$ from $B$ into $A$. The new partition is easily seen to be nice. Since $u_3\ne u_5$, the vertex $v_3$ has all its neighbors in $A$, and therefore $v_2$ is unstable.

Now we may assume that $G$ has no 5-cycles. Since every vertex belongs to a cycle of length at most 5, we see that every vertex lies in a triangle, which gives two adjacent triangles. This case was treated before, and thus the proof is complete.
\end{proof}

\section{Large graphs have many small eigenvalues}
\label{sect:large}

The proof of Theorem \ref{thm:cubic} can be used to derive an even stronger conclusion, asserting that a positive fraction of the ``median'' eigenvalues lie between $-\sqrt{2}$ and $\sqrt{2}$.

\begin{theorem}
\label{thm:linearly many}
There is a constant $\delta>0$ such that for every subcubic graph $G$ of order $n$, all its eigenvalues $\l_i(G)$, where $(\tfrac{1}{2}-\delta)n\le i\le (\tfrac{1}{2}+\delta)n$, belong to the interval $[-\sqrt{2},\sqrt{2}]$.
\end{theorem}

\begin{proof}
We shall prove the theorem with $\delta=1/6140$, with no intention to improve the constant.

Let $G$ be a subcubic graph and $v\in V(G)$. Then one of the following holds:
\begin{itemize}
  \item[a)]
  There is a vertex $u$ at distance at most 3 from $v$ such that $\deg(u)\le 2$.
  \item[b)]
  $v$ does not belong to a cycle of length $5$ or less.
  \item[c)]
  $v$ lies on an induced cycle of length 4 or 5 whose vertices have distinct neighbors out of the cycle.
  \item[d)]
  $v$ or one of its neighbors lies on a 3-cycle which either shares an edge with another 3-cycle or one of its vertices is adjacent to another 3-cycle.
  \item[e)]
  $v$ lies on a 4-cycle or a 5-cycle with an adjacent triangle.
\end{itemize}
In each of these cases, when previous cases were excluded for the vertex $v$, we found a set $D$ of vertices, all at distance at most 4 from $v$, and a subset $C\subseteq D$, and have defined a multigraph $G'$ by deleting vertices in $C$ and adding some edges between the vertices in $D\setminus C$. After that we proved that, starting with an unfriendly partition of $G'$, we can extend it to a nice partition $\{A,B\}$ of $G$ such that one of the following holds:
\begin{itemize}
  \item[(i)]
  $\{A,B\}$ has an unstable vertex in $D$, or
  \item[(ii)]
  $\{A,B\}$ can be changed by exchanging parts for some vertices in $D$ such that some vertex in $D$ becomes unstable, or
  \item[(iii)]
  $\{A,B\}$ can be changed to a partition $\{\hat A,\hat B\}$, where either $\hat A\subseteq A$, $\hat B \subseteq B\cup D$, $|\hat B|=|B|+2k$ (where $k\in \{2,3\}$), and all components of $G(\hat B)$ have at most two edges except for the component $Q$ containing vertices in $C$, and $Q$ satisfies $\l_k(Q)\le \sqrt{2}$ and $\l_k^-(Q)\ge -\sqrt{2}$; or the same properties hold with the roles of $\hat A$ and $\hat B$ interchanged.
\end{itemize}

Now, let us start with a maximal collection of vertices $v_1,\dots,v_t\in V(G)$ such that any two of them are at distance at least 11 from each other. Since every vertex $v$ in a subcubic graph has at most $1+3+3\cdot 2+3\cdot 2^2+\cdots +3\cdot 2^9 = 3070$ vertices at distance 10 or less, we have that $t\ge n/3070$.

For each vertex $v_i$ ($1\le i\le t$) we consider the corresponding subcase a)--e) and let $C_i\subseteq D_i$, $Q_i$ and $k_i\in \{2,3\}$ (in case (iii)) be the corresponding quantities from the above description. We define $k_i=1$ in cases (i) and (ii). Now we form the multigraph $G'$ by removing all subsets $C_i$ ($1\le i\le t$) and adding edges to each $D_i\setminus C_i$ as required in each case. The distance condition implies that these changes do not interfere with each other, neither when defining $G'$, nor when producing the nice partition $\{A,B\}$ or $\{\hat A,\hat B\}$ satisfying (i)--(iii) simultaneously for each case.

Each change in (i)--(iii) increases either $|A|$ or $|B|$ by $k_i$, where $k_i\in \{1,2,3\}$. Let $a$ be the number of cases where $|A|$ increases, and let $b=t-a$ be the number of cases where $|B|$ increases. Then either $|A|+a\ge (|B|-a)+t$ or $|B|+b\ge (|A|-b)+t$. Whichever case holds, we can get a partition of $V(G)$ to which the eigenvalue interlacing theorem can be applied to show that $t/2$ eigenvalues preceding $\l_H(G)$ and $t/2$ eigenvalues succeeding $\l_L(G)$ all belong to the interval $[-\sqrt{2},\sqrt{2}]$.
\end{proof}

Note that $\sqrt{2}$ in Theorem \ref{thm:linearly many} cannot be replaced by any smaller number as shown by the union of many copies of the Fano plane incidence graph.

As for a ``rough converse'' of Theorem \ref{thm:linearly many}, let us observe that every subcubic graph $G$ of order $n$, none of whose components is a path of length at most 3, has at least $n/30$ eigenvalues that are larger or equal to $\sqrt{3}$. To see this, observe that $G$ has an induced subgraph consisting of $t\ge n/30$ copies of $K_{1,3}$, $K_3$, and $P_5$. Since the spectral radius of each of these components is at least $\sqrt{3}$, eigenvalue interlacing implies the above stated fact. The following theorem shows that the constant $\sqrt{3}$ can be improved to $2-\varepsilon$ if $G$ has minimum degree at least 2, and to $2\sqrt{2}-\varepsilon$ if $G$ is cubic.

\begin{theorem}
\label{thm:linearly many large}
Let $\Delta\ge\delta\ge2$ be integers and let $\varepsilon>0$ be a real number. Then there is a constant $c=c(\Delta,\varepsilon)$ such that every graph of order $n$, of maximum degree at most $\Delta$ and minimum degree at least $\delta$ has at least $\lceil n/c\rceil$ eigenvalues that are larger than $2\sqrt{\delta-1}-\varepsilon$.
\end{theorem}

\begin{proof}
Let $G$ be a graph of minimum degree at least $\delta$ and let $v\in V(G)$. Let $B_r(v)$ (the \DEF{$r$-ball} around $v$) be the induced subgraph of $G$ on all vertices whose distance from $v$ is at most $r$. It is shown in \cite{Mo} that
$$
    \l_1(B_r(v)) > 2\sqrt{\delta-1}\,\Bigl(1 - \frac{\pi^2}{r^2} + O(r^{-3})\Bigr).
$$
We take $r$ large enough so that the right hand side quantity in the above inequality is at least $2\sqrt{\delta-1}-\varepsilon$. By taking a maximum subset $U$ of vertices that are at distance at least $r+2$ from each other, we see that their $r$-balls are disjoint and non-adjacent, and the interlacing theorem shows that $G$ has at least $|U|$ eigenvalues that are larger than $2\sqrt{\delta-1}-\varepsilon$.
\end{proof}

Cycles show that $2-\varepsilon$ cannot be improved in the subcubic case (with minimum degree 2), and Ramanujan graphs \cite{LPS,Mar} show that $2\sqrt{2}-\varepsilon$ cannot be increased for cubic graphs.

A version of Theorem \ref{thm:linearly many large} for negative eigenvalues does not hold since there are $d$-regular graphs of arbitrarily large degree and diameter whose smallest eigenvalue is $-2$ (for example, all line graphs have this property).

We refer to \cite{Mo} for some related results.

\section{Open problems}

We are not aware of any planar subcubic graph with $R(G)>1$ and have some evidence for the following speculation.

\begin{conjecture}
If\/ $G$ is a planar subcubic graph, then $R(G)\le1$.
\end{conjecture}

The following open problems remain unanswered:

\begin{enumerate}
\item[(1)]
It remains to see if $\sqrt{d}$ or $\sqrt{d-1}$ is the correct value for $R(d)$.
It could be either one of these two, or some value strictly between them. However, we believe that the lower bound is the correct value.
\item[(2)]
Is $R(d)=\widehat R(d)$ for every integer $d\ge0$?
\item[(3)]
If $G$ is a subcubic graph of odd order, what is the maximum value of $|\l_{H-1}|$ and $|\l_{H+1}|$? Examples of $C_3$ and $C_5$ show that $|\l_{H-1}|$ can be as large as 2 and $\l_{H+1}$ can be as small as $-2\cos\tfrac{4\pi}{5} \approx -1.618$. However, Theorem \ref{thm:linearly many} implies that there are only finitely many examples for which either $|\l_{H-1}|>\sqrt{2}$ or $|\l_{H+1}|>\sqrt{2}$.
\item[(4)]
The proof of Theorem \ref{thm:average degree} in Section \ref{sect:2} suggests that extremal graphs (e.g. those having maximum HL-index among $d$-regular graphs) must be close to be strongly regular. A similar property holds for graphs with maximum energy as shown in \cite{Ha,KoMo}. However, it turns out that the graphs that are energy-extremal are not extremal for the HL-index. It remains an open problem to determine extremal examples for the HL-index.
\end{enumerate}



\end{document}